\documentclass[11pt]{amsart}
\usepackage{amsfonts,amssymb,amsmath,amsthm}
\usepackage{url}
\usepackage{enumerate}
\usepackage[all]{xy}

\usepackage{amsmath,amsfonts,amsthm,url,color,amssymb}
\usepackage{graphicx}
\usepackage{algpseudocode, algorithm}
\usepackage{hyperref}
\usepackage{todonotes}

\usepackage[norefs,nocites]{refcheck}

\newcommand{\m}{\mathcal M}
\newcommand{\rad}{\mathrm{rad}}

\newcommand{\F}{\mathbb{F}}
\newcommand{\pp}{\mathcal P}

\newtheorem{theorem}{Theorem}[section]
\newtheorem{lemma}[theorem]{Lemma}
\newtheorem{conjecture}[theorem]{Conjecture}

\newtheorem{proposition}[theorem]{Proposition}
\theoremstyle{definition}
\newtheorem{definition}[theorem]{Definition}

\theoremstyle{remark}
\usepackage{amssymb}
\newtheorem{remark}[theorem]{Remark}
\usepackage{enumerate}
\author{Herivelto Borges}
\address{Universidade de S\~{a}o Paulo, Instituto de Ci\^{e}ncias Matem\'{a}ticas e de Computa\c{c}\~{a}o, S\~{a}o
Carlos, SP 13560-970, Brazil}
\curraddr{}
\email{hborges@icmc.usp.br}
\thanks{The first  author was supported by CNPq
(Brazil), grants 406377/2021-9 and 310194/2022-9, and by FAPESP (Brazil), grant 2022/15301-52}
\thanks{}

\author{Lucas Reis}
\address{Departamento de Matem\'{a}tica, Universidade Federal de Minas Gerais, UFMG, Belo Horizonte, MG, 30123-970, Brazil.}
\curraddr{}
\email{lucasreismat@mat.ufmg.br}
\thanks{The second author was supported by CNPq (Brazil), grant 309844/2021-5 and by FAPEMIG (Brazil), grant APQ-01712-23.}

\begin{document}

\title{Minimal value set polynomials  } 
\begin{abstract}A well-known problem in the theory of polynomials over finite fields is the characterization of minimal value set polynomials (MVSPs) over the finite field $\mathbb{F}_q$, where $q = p^n$. These are the nonconstant polynomials $F \in \mathbb{F}_q[x]$ whose value set $V_F = \{F(a) : a \in \mathbb{F}_q\}$ has the smallest possible size, namely $\lceil \frac{q}{\deg(F)} \rceil$. 
In this paper, we describe the family $\mathcal{A}_q$ of all subsets $S \subseteq \mathbb{F}_q$ with $\# S>2$ that can be realized as the value set of an MVSP $F \in \mathbb{F}_q[x]$. Affine subspaces of $\mathbb{F}_q$ are a fundamental type of set in $\mathcal{A}_q$, and we provide the complete list of all MVSPs with such value sets. Building on this, we present a conjecture that characterizes all MVSPs $F \in \mathbb{F}_q[x]$ with $V_F=S$ for any $S \in \mathcal{A}_q$. The conjecture  is  confirmed by prior results for $q \in\left\{p, p^2, p^3\right\}$ or $\# S \geq p^{n / 2}$, and  additional instances, including the cases for $q=p^4$  and  $\# S>p^{n / 2-1}$, are proved. We further show  that the conjecture leads to the complete characterization
of the $\F_{q}$-Frobenius nonclassical curves of type $y^d=f(x)$,  which we establish as a theorem for $q=p^4$. 
\end{abstract}

\maketitle

\section{Introduction}
Let $q=p^n$ be a power of a prime $p$, and let $\mathbb{F}_q$ be the finite field with $q$ elements. A polynomial $F \in \mathbb{F}_q[x]$ of positive degree $\operatorname{deg} F$ is called a minimal value set polynomial (MVSP) if its value set $V_F:=\left\{F(\alpha): \alpha \in \mathbb{F}_q\right\}$ has the least possible size $\# V_F=\lfloor(q-1) / \operatorname{deg} F\rfloor+1$. The MVSPs were introduced by Carlitz, Lewis, Mills, and Straus in the early 1960s, mainly motivated by a generalization of Waring's problem modulo a prime. In  ~\cite{Carlitz, Mills}, they provided significant results on the theory of MVSPs, culminating in the complete characterization of such polynomials when $q\in \{p, p^2\}$. Moreover,  in~\cite{Mills}, Mills suggests that all MVSPs fall into two main families of polynomials, although he did not pose this as a conjecture. Since these two early works, several authors have studied polynomials over finite fields with small value sets \cite{Bor2,Cald1,Cald2,WSC}. 
In \cite{Bor2},  the authors described all MVSPs in $\mathbb{F}_q[x]$ whose value set is a subfield of $\mathbb{F}_q$. As a consequence, they found MVSPs that Mills suspected did not exist \cite[Section $6$]{Bor2}. In ~\cite{Bor0}, the author established  a detailed  connection between minimal value set polynomials in $\mathbb{F}_q[x]$ and   $\mathbb{F}_q$-Frobenius nonclassical curves $\mathcal{X} \subseteq \mathbb{P}^2$. These curves, introduced by St\"ohr and Voloch in~\cite{SV},  are known for their special geometric and arithmetic features, being particularly useful for applications in coding theory and finite geometry ~\cite{Bor1,Bor3}. The connection between MVSPs and $\mathbb{F}_q$-Frobenius nonclassical curves reveals that the condition for the irreducible components of the curve with separated variables
$$
F(x)-G(y)=0
$$
to be $\mathbb{F}_q$-Frobenius nonclassical curves is equivalent to $F$ and $G$ being MVSPs in $\mathbb{F}_q[x]$ with the same value set. Recently, all MVSPs over fields of size $p^3$ were characterized in \cite{BR}, thus  completing  the  list of all $\mathbb{F}_{p^{3}}$-Frobenius nonclassical curves that are cyclic covers of $\mathbb{P}^1$ .

Despite the relevance of MVSPs and the existing results on such polynomials, their comprehensive characterization remains elusive. There are two primary approaches concerning their characterization. One aims to describe these polynomials over finite fields $\mathbb{F}_{p^k}$ for all $k \geq 1$, while the other seeks to characterize the MVSPs $F \in \mathbb{F}_q[x]$ such that $V_F=S$  for each nonempty set $S \subseteq \mathbb{F}_q$.  These approaches are interconnected; for instance, in the complete characterization of MVSPs over  $\mathbb{F}_{p^3}$, we  relied heavily on the characterization of MVSPs $F \in \mathbb{F}_{p^3}[x]$ with $V_F=\mathbb{F}_p$ as provided in \cite{Bor2}.

This work is founded on the following questions.

\begin{enumerate}[\rm(1)]
\item What is the complete list of nonempty sets $S \subseteq \mathbb{F}_q$ for which there exist  MVSPs  $F \in \mathbb{F}_q[x]$ with the value set $V_F=S$?
\item For each such $S$, characterize the set $\mathcal{P}(S, q)$ of all MVSPs $F \in \mathbb{F}_q[x]$ with $V_F=S$.
\end{enumerate}

As noted in \cite{Carlitz}, the cases where $\# S \leq 2$ do not adhere to the general pattern, but are relatively easy to characterize. Specifically, $\# V_F=1$ if and only if $F \in \mathbb{F}_q[x]$ is of the form
$$
F(x)=\alpha+\left(x^q-x\right) G(x),
$$
where $G \in \mathbb{F}_q[x]$ is non-zero. For the case $\# V_F=2$, it suffices to recall that for any two distinct elements $\alpha, \beta \in \mathbb{F}_q$, and any nonempty subset $U \subseteq \mathbb{F}_q$, Lagrange interpolation yields the unique polynomial
$$
F(x)=\alpha \sum_{\gamma \in U}\left(1-(x-\gamma)^{q-1}\right)+\beta \sum_{\gamma \notin U}\left(1-(x-\gamma)^{q-1}\right)
$$
satisfying $F(\gamma)=\alpha$ for $\gamma \in U$, and $F(\gamma)=\beta$ for $\gamma \notin U$. Consequently, our discussion will focus on MVSPs $F \in \mathbb{F}_q[x]$ for which $\# V_F>2$.

In this paper, we first answer question (1) by showing that all such sets $S$ arise from $\mathcal{U}^v = \{u^v : u \in \mathcal{U}\}$, where $\mathcal{U}$ is an $\mathbb{F}_{p^k}$-vector space with $\mathbb{F}_{p^k} \subseteq \mathcal{U}   \subseteq  \mathbb{F}_q$ and $v \mid (p^k - 1)$. We then characterize the sets $\mathcal{P}(S, q)$ in question (2) when $S$  is an affine space $a \cdot \mathcal{U} + b \subseteq  \mathbb{F}_q$.
This generalizes the main result of \cite{Bor2}, where the authors address the subfield  case  $\mathcal{U}=\mathbb{F}_{p^e} \subseteq  \mathbb{F}_q$.
%
%
%
%

In addition, improving the results in \cite{Mills}, we characterize $\mathcal{P}(S, q)$ when $\# S>p^{n / 2-1}$. Furthermore, extending the techniques employed in \cite{BR}, we provide a complete characterization of all MVSPs in $\mathbb{F}_{p^4}[x]$. Based on this, we propose the following conjecture, which asserts that all MVSPs are essentially derived from suitable compositions of polynomials belonging to well-characterized families of MVSPs.

\begin{conjecture}\label{conj}
Let $q=p^n$, $k$ a divisor of $n$, and $v\geq 1$ a divisor of $p^k-1$.
Let $\mathcal{U} \subseteq \mathbb{F}_q$ be an $\mathbb{F}_{p^k}$-vector space of dimension $m$ such that $1 \in \mathcal{U}$ and
$\# \mathcal{U}^v>2$. If $\mathcal{U}$ is not a field, then $\mathcal{P}\left(\mathcal{U}^v, q\right)=\left\{f^v \mid f \in \mathcal{P}(\mathcal{U}, q)\right\}$. Otherwise,
$$
\mathcal{P}\left(\mathcal{U}^v, q\right)=\left\{\left.f^{\frac{\left(p^e-1\right) v}{p^{m k}-1}} \right\rvert\, f \in \mathcal{P}\left(\mathbb{F}_{p^e}, q\right)\right\},
$$
where $\mathbb{F}_{p^e}  \subseteq \F_q$ is the smallest field containing $\mathcal{U}^v$.
\end{conjecture}

This conjecture is confirmed by all previous results for $q \in\left\{p, p^2, p^3\right\}$, and further evidence is given by the new results established in this paper.


\section{The results}

\begin{definition}
Let $S\subseteq \F_{q}$ be a nonempty subset.
\begin{enumerate}[\rm(i)]
\item  The set of all MVSPs $F\in \F_q[x]$ with $V_F=S$ is denoted by $\pp(S, q)$.
\item For any integer $v\geq 1$, we define $ S^v=\{s^v\, : s \in  S \}$.
\end{enumerate}
\end{definition}

The following result answers question (1) presented in the Introduction.

\begin{theorem}\label{thm:main1}
Let $q = p^n$ and $S \subseteq \mathbb{F}_q$, where $\#S > 2$. Then $\pp(S, q)\ne \emptyset$ if and only if 
$$S = a \cdot \mathcal{U}^v + b,$$
where
\begin{enumerate}[\rm(i)]
\item  $a, b \in \mathbb{F}_q$, with $a \neq 0$;
\item $\mathcal{U}$ is an $\mathbb{F}_{p^k}$-subspace of $\mathbb{F}_q$, with $1 \in \mathcal{U}$ and $k\geq1$;
\item  $v$ divides $p^k-1. $
\end{enumerate}
\end{theorem}

\begin{remark}\label{rem2.3}
We observe that if $\mathcal{U}$ is an $\mathbb{F}_{p^k}$-subspace of $\mathbb{F}_q$ of  dimension $m$, and $v$ is a divisor of $p^k-1$, then the set $\mathcal U^v$ has cardinality $\frac{p^{mk}-1}{v}+1$. In particular, $\#\mathcal U^v>2$ if and only if $(m, v)\ne (1, p^k-1)$. 
\end{remark}

Note that for every $S\subseteq \F_q$ and  $a, b\in \F_q$, with $a\ne 0$, we have the equality 
$$\mathcal P(a\cdot S+b, q)=\{af(x)+b\,|\, f\in \mathcal P(S, q)\},$$
and after  Theorem~\ref{thm:main1}, the  study  of $\mathcal P(S, q)$ can be reduced to the case $S=\mathcal U^v$, with $\mathcal U$ and $v$ as in Theorem~\ref{thm:main1}.   In particular, a complete characterization of MVSPs over finite fields is obtained once we describe $\pp(\mathcal U^v, q)$ for every admissible pair $(\mathcal U, v)$. 

The case where $v=1$ and $\mathcal U$ is a subfield of $\F_q$ is  completed in~\cite[Theorem 4.7]{Bor2}. In the following theorem, we extend this classification  to the case where $\mathcal U$ is a generic vector subspace of $\F_q$. More precisely, we provide a complete characterization of  $\mathcal{P}(S, q)$ for $S = \mathcal{U}$, i.e., $v=1$, in terms of the sets $\mathcal{P}(\F_{p^e}, q)$,  with $e | n$, already
characterized in~\cite{Bor2}.  Before stating the theorem,   let us recall that, for a positive integer $k$, a polynomial is called  $p^k$-linearized if it is of the form $\sum_{i=0}^s a_i x^{p^{ki}} \in \overline{\mathbb{F}}_q[x]$. 

\begin{theorem}\label{thm:main2}
Let $q=p^n$ and $k$ be a divisor of $n$. Let  $\mathcal U\subseteq \F_q$ be an $\F_{p^k}$-vector space such that  $1 \in \mathcal U$ and $\# \mathcal U>2$, and  let $d\leq n/k$ be  the smallest positive  integer for which  $ \mathcal U\subseteq \F_{p^{dk}}$. If $A(x)=\prod_{u\in \mathcal U}(x- u)$, then there exists a monic $p^k$-linearized polynomial $M\in \F_q[x]$ such that $A(M(x))=M(A(x))=x^{p^{dk}}-x$. Moreover, $\F_{p^{dk}}\subseteq \F_q$ and we have
$$\mathcal P(\mathcal U, q)=\left\{ M(f(x))\,|\, f\in \pp(\F_{p^{dk}}, q)\right\}.$$
\end{theorem}
There are three main families of MVSPs over $\F_q$: 
\begin{enumerate}[\rm(i)]
\item the monomials $x^v$, where $v$ divides  $q-1$
\item  the  $p$-linearized polynomials $L\in \F_q[x]$ that divide $x^q-x$
\item the MVSPs $F\in \F_q[x]$ whose value set is a subfield of $\F_q$.
\end{enumerate}
Some compositions of such MVSPs also generate new MVSPs, e.g., the family of MVSPs in Theorem~\ref{thm:main2}. Bearing this in mind,  Conjecture 1.1 states  that all MVSPs are essentially obtained from suitable compositions of polynomials coming from these three families. Its  proof, combined with Theorem~\ref{thm:main2}, will complete the characterization of all minimal value set polynomials.

\begin{remark}\label{rem} Conjecture \ref{conj} can be easily verified whenever  $\mathcal U^v=\F_{p^d}$ is a subfield of $\mathbb{F}_q$. Indeed, from Remark \ref{rem2.3},  $\mathcal U^v$ has cardinality $\frac{p^{mk}-1}{v}+1=p^d$, and then  $v=\frac{p^{m k}-1}{p^d-1}$. If $v=1$,  then $\mathcal{U}^v=\mathcal{U}=\mathbb{F}_{p^{m k}}$, and then we are in the second case of Conjecture 1.1 with $e=m k$, which gives $\frac{\left(p^e-1\right) v}{p^{m k}-1}=1$, and thus the statement holds. For $v>1$, since $v$ divides $p^k-1$, we have that $d \leq m k / 2$ and $\frac{p^{m k}-1}{p^d-1} \leq p^k-1$. Thus $(m-1) k<d \leq m k / 2$, which implies that $m=1$, and then $\mathcal{U}=\mathbb{F}_{p^k}$. We are in the second case of Conjecture \ref{conj} with $e=d$ and $\frac{\left(p^e-1\right) v}{p^{m k}-1}=1$, and then the statement holds trivially.
\end{remark}

The following theorem confirms additional  cases of Conjecture \ref{conj}.

\begin{theorem}\label{pre-conj}  Let $q=p^n$ be a prime power, $k$ a divisor of $n, v>1$  a divisor of $p^k-1$, and let $\mathcal{U} \subseteq \mathbb{F}_q$ be an $\mathbb{F}_{p^k}$-vector space of dimension $m$ such that $\# \mathcal{U}^v>2$ and $1 \in \mathcal{U}$. Then the following hold.
\begin{enumerate}[\rm(i)]
\item If $\mathcal{U} \neq \mathbb{F}_{p^{m k}}$, then $m>1$ and $\operatorname{deg}(F)>\left(p^{n / 2}+1\right) p^{m k}$ for every
$$
F \in \mathcal{P}\left(\mathcal{U}^v, q\right) \backslash\left\{f^v \mid f \in \mathcal{P}(\mathcal{U}, q)\right\} .
$$
\item  If $\mathcal{U}=\mathbb{F}_{p^{m k}}$, let e be the smallest positive integer such that $\mathcal{U}^v \subseteq \mathbb{F}_{p^e} \subseteq$ $\mathbb{F}_{p^{m k}}$, and set $t=\frac{v\left(p^e-1\right)}{p^{m k}-1} $. In this case, if
$$
F \in \mathcal{P}\left(\mathcal{U}^v, q\right) \backslash\left\{f^t \mid f \in \mathcal{P}\left(\mathbb{F}_{p^e}, q\right)\right\},
$$
then $t>1$, $\mathcal{U}^v$ is not a subfield of $\mathbb{F}_q$, and $\operatorname{deg}(F) \geq\left(p^{n / 2}+1\right) p^e$.
 \end{enumerate}
In particular, Conjecture  \ref{conj} holds in each of the following cases:
\begin{enumerate}[\rm(a)]
\item $\# \mathcal{U}^v>p^{n / 2-1}$;
\item  $m>1$ and $(2 m-1) k \geq n / 2$.

 \end{enumerate}

\end{theorem}

A consequence of our previous results is that  Conjecture~\ref{conj} holds for $q=p^4$, which will lead us to the complete classification of MVSPs in $\mathbb{F}_{p^4}[x]$, as given by the following theorem.

\begin{theorem}\label{thm:p4}
Up to pre- and post-compositions with affine maps $x \mapsto ax + b$, a polynomial $F \in \mathbb{F}_{p^4}[x]$ is an MVSP with value set $V_F$ satisfying $r := |V_F| - 1 \geq 2$ if and only if $F$ is of one of the following types:
\begin{enumerate}[\rm(i)]
\item $x^v$, where $v$ is a proper  divisor of $p^4 - 1$;
\item $g(x)^t$, where $t$ is a proper divisor of $p - 1$, and $g \in \mathbb{F}_{p^4}[x]$ is an MVSP with $V_g = \mathbb{F}_{p}$, i.e., $g$ is a nonconstant polynomial of the form
\begin{align*}
& ax^{p^3+p^2+p+1} + bx^{p^3+p^2+p} + b^p x^{p^3+p^2+1} + b^{p^2} x^{p^3+p+1} + b^{p^3} x^{p^2+p+1} \\
& + cx^{p^3+p^2} + c^p x^{p^3+1} + c^{p^2} x^{p+1} + c^{p^3} x^{p^2+p} + dx^{p^3+p} + d^p x^{p^2+1} \\
& + ex^{p^3} + e^p x + e^{p^2} x^{p} + e^{p^3} x^{p^2} + f,
\end{align*}
where $a, f \in \mathbb{F}_p$, $d \in \mathbb{F}_{p^2}$, and $b, c, e \in \mathbb{F}_{p^4}$;
\item $g(x)^t$, where $t$ is a proper  divisor of $p^2 - 1$, and $g$ is an MVSP with $V_g = \mathbb{F}_{p^2}$, i.e., $g$ is a nonconstant polynomial of the form
$$ax^{p^2+1} + bx^{p^2} + b^{p^2} x + c,$$
 where $a, c \in \mathbb{F}_{p^2}$ and $b \in \mathbb{F}_{p^4}$;
\item $(x^{p^2} + (1 + (\beta^{p^3} - \beta^{p^2})^{p-1})x^p - (\beta^p - \beta)^{1-p}x)^v$, where $\beta \in \mathbb{F}_{p^4} \setminus \mathbb{F}_p$ and $v$ divides $p-1$;
\item $(x^p - x)^v$, where $v$ divides $p-1$.
\end{enumerate}
\end{theorem}

Another consequence of the previous results is the characterization of $\mathbb{F}_{p^4}$-Frobenius nonclassical curves of the form $y^d = f(x)$. This follows from the fact that Conjecture~\ref{conj} holds for $q = p^4$, along with the following general  result, which is conditional on Conjecture~\ref{conj}. Note that Theorem~\ref{FNC} below benefits from our complete understanding of the sets  $\mathcal{P}\left(\mathbb{F}_{p^e}, q\right)$.

\begin{theorem}\label{FNC}
Suppose Conjecture~\ref{conj} holds. Let $q = p^n$ be a prime power, and let $\mathcal{F}: y^d = f(x)$ be an irreducible plane curve defined over the finite field $\mathbb{F}_q$. If $d < q - 1$, then $\mathcal{F}$ is $\mathbb{F}_q$-Frobenius nonclassical if and only if  
$d = \frac{q - 1}{p^e - 1}$, where $e \mid n$, and $f(x)$ is an MVSP whose value set is $\mathbb{F}_{p^e}$.
\end{theorem}

%
%
%
%

%
%

\section{Preparation}
We begin by presenting a few  technical results that will be used later on.

\begin{theorem}\cite[Theorem 5.41]{Nie} \label{thm:Weil} Let $\chi$ be a multiplicative character of $\F_{q}$ of order $r>1$ and  $F\in \F_{q}[x]$ be a nonconstant polynomial such that $F$ is not of the form $ag(x)^r$ with $a\in \F_q$ and $g\in \F_{q}[x]$.  Suppose that $z$ is the number of distinct roots of $F$ in its splitting field over $\F_{q}$. Then
$$\left|\sum_{c\in \F_{q}}\chi(F(c))\right|\le (z-1)\sqrt{q}.$$
\end{theorem}

\begin{theorem}[Mason-Stothers]\label{abc} Let $k$ be an algebraically closed field and $a, b, c\in k[x]$ pairwise relatively prime polynomials such that $a+b=c$. If the derivatives of these polynomials are not all vanishing, then 
$$\max\{\deg(a), \deg(b), \deg(c)\}\le \deg(\rad(abc))-1,$$
where $\rad(abc)$ denotes the square-free part of $abc$.
\end{theorem}

\begin{lemma}\label{lem:linear}
Let $L, L_1\in \F_{q}[x]$ be $p^k$-linearized polynomials and  $\mathcal U\subseteq \F_q$ be an $\F_{p^k}$-vector space of dimension $m\ge 1$. Then the following hold.

\begin{enumerate}[\rm(i)]
\item The polynomial $\prod_{u\in \mathcal U}(x-u)$ is $p^k$-linearized of degree $p^{mk}$.
\item The map $y\mapsto L(y)$ from $\F_q$ to itself is $\F_{p^k}$-linear.
\item $L$ divides $L_1$ if and only if $L_1(x)=M(L(x))$ for some $p^k$-linearized polynomial $M\in \F_q[x]$. Moreover, if $L_1(x)=x^{p^{kd}}-x$, $d\geq 1,$ then $M(L(x))=L(M(x))$. 
\end{enumerate}
\end{lemma}

\begin{proof}
Assertion  (i) follows  from \cite[Theorem 3.52]{Nie}, and  (ii) follows directly from  (i). 
Finally, the first part of (iii) follows from  \cite[Exercise 3.68]{Nie}, and,  for the second part, see the proof of   \cite[Lemma 3.4]{R21}. 
\end{proof}

\begin{lemma}\label{novo}
Let $n$ be a positive integer and $B(x)=ax^{{q^i}}+bx^{q^{j}}\in \F_{q^n}[x]$, with $a, b\ne 0$ and $i>j\ge 0$. For each $m\ge i$, there exist an integer $s_m$, with $0\le s_m<i$, an element $c_m\in \F_{q^n}^*$, and a $q$-linearized polynomial $T_m\in \F_{q^n}[x]$ such that 
$$x^{q^{m}}=B(T_{m}(x))+c_mx^{q^{s_m}}.$$
\end{lemma}
\begin{proof}
We proceed by induction on $m$. If $m=i$, we can take $T_i(x)=a^{-q^{-i}}x$, $s_i=j$,  and $c_i=-ba^{q^{j-i}}$. Suppose the result holds for some  integer $M \geq i$, i.e., 
$x^{q^{M}}=B(T_{M}(x))+c_Mx^{q^{s_M}}$. Replacing $x$ with $x^q$ yields
\begin{equation}\label{step1}
x^{q^{M+1}}=B(T_{M}(x^q))+c_Mx^{q^{s_M+1}}.
\end{equation}
If $s_M+1<i$,  then we are finished. Otherwise,
we must have  $s_M+1=i$, as  $s_M<i$. In this case, we consider the base-step 
$x^{q^{i}}=B(T_{i}(x))+c_ix^{q^{s_i}},$ and  replace $x$ with $\varepsilon x$,  where  $\varepsilon^{q^i}=c_M$, obtaining 
\begin{equation}\label{step2}
c_Mx^{q^{i}}=B(T_i(\varepsilon x))+c_i\varepsilon^{q^{s_i}}x^{q^{s_i}}.
\end{equation}
Since $s_M+1=i$, combining Eqs. \eqref{step1} and \eqref{step2} entails
$$x^{q^{M+1}}=B(T_{M}(x^q))+B(T_i(\varepsilon x))+c_i\varepsilon^{q^{s_i}}x^{q^{s_i}},$$
and thus  defining  $T_{M+1}(x)=T_M(x^q)+T_i(\varepsilon x)$, $c_{M+1}=c_i\varepsilon^{q^{s_i}}$, and $s_{M+1}=s_i$, the proof is concluded.
\end{proof}

Next, we present some preliminary results on minimal value set polynomials. The following is proved in \cite[Theorem 1]{Mills}. 

\begin{theorem}[Mills]\label{thm:mills}
Let $F\in \F_q[x]$ be an MVSP with value set $V_{F}=\{\gamma_0, \ldots, \gamma_r\}$, where $\# V_F=r+1>2$ and $\# (F^{-1}(\gamma_0)\cap \F_q)\le \# (F^{-1}(\gamma_i)\cap \F_q)$ for  $i=1,\ldots, r$. Set $L=\gcd(F-\gamma_0,x^q-x)$  and $T=\prod\limits_{i=0}^{r}(x-\gamma_i)$. Then there exist positive integers $m, k, v$  and polynomials $A, B, N\in \F_q[x]$, where $L\nmid N$  such that the following hold.
\begin{enumerate}[\rm(i)]
\item $\F_{p^k} \subseteq \F_q$, $v$ divides $p^{k}-1$, and $vr+1=p^{mk}$.
\item $F=L^vN^{p^{mk}}+\gamma_0$.
\item $L=A^{p^{mk}}x+B^p$.
\item $T(x+\gamma_0)/x=\sum\limits_{i=0}^{m}w_ix^{\frac{p^{ki}-1}{v}}  \in \F_q[x]$,  with $w_m=1$ and $w_0\ne 0$.
\end{enumerate}
\end{theorem}

The following result  is proved in  \cite[Theorem 3.1]{Bor2},  and in \cite[Lemma 2.4]{Bor1}.

\begin{theorem}\label{MillsBor}
Let $F \in \F_q[x]$ be a polynomial with value set $V_{F}=\{\gamma_0, \ldots, \gamma_r\}$, where $r>1$.
Then $F$ is an MVSP if and only if there exist a monic polynomial $T \in \F_q [x]$ and $\theta \in \F_q^*$ such that
\begin{equation}\label{MVSP}
T(F)=\theta(x^q-x)F^{\prime}.
\end{equation}
Furthermore, the polynomial $T$ is uniquely determined as $T=\prod\limits_{i=0}^{r}(x-\gamma_i)$, and if $\gamma_0, \ldots, \gamma_r$ are ordered as in Theorem
\ref{thm:mills}, then $\theta=-T^{\prime}\left(\gamma_{i}\right)$ for all $\gamma_{i} \in V_{F}\backslash {\gamma_0}$.
\end{theorem}

\subsection{Frobenius Nonclassical Curves}

An irreducible plane curve $\mathcal{C}$, defined over $\mathbb{F}_q$, is called $\mathbb{F}_q$-Frobenius nonclassical if the $\mathbb{F}_q$-Frobenius map takes each nonsingular point $P \in \mathcal{C}$ to the tangent line to $\mathcal{C}$ at $P$. These curves, introduced by St\"ohr and Voloch in \cite{SV}, possess important geometric and arithmetic properties, making their characterization highly desirable. For instance, $\mathbb{F}_q$-Frobenius nonclassical curves of degree $d$ are known to have at least $d( q+2 - d)$ $\mathbb{F}_q$-rational points, as shown in \cite[Theorem 1.3]{BoHo}. A well-known example of an $\mathbb{F}_{q^2}$-Frobenius nonclassical curve is the Hermitian curve defined by the equation $y^{q+1} = x^q + x$. For further details, see \cite[Chapter 8]{HKT} and \cite{SV}.

Despite the few known examples in the literature, the complete characterization of Frobenius nonclassical curves remains an open problem. Partial characterizations of $\mathbb{F}_{p^k}$-Frobenius nonclassical curves of type  $y^d = f(x)$ have been provided,  particularly for the cases corresponding to $k \leq 3$ \cite{Bor0, BR, Garcia}. A consequence of our  results  on the classification of MVSPs is the characterization of  all $\mathbb{F}_{p^4}$-Frobenius nonclassical curves  of type  $y^d = f(x)$.  Furthermore,  Theorem~\ref{FNC} establishes that 
Conjecture~\ref{conj}, once   proved,  entails   the  complete characterization of all $\mathbb{F}_{p^k}$-Frobenius nonclassical curves  of type  $y^d = f(x)$. 

\section{Proof of Theorem \ref{thm:main1}}

The proof of Theorem  \ref{thm:main1}  is divided into two parts. First, we prove that the conditions on $S$ in Theorem  \ref{thm:main1} are necessary.

\begin{lemma}\label{cor:vec}
Let $F\in \F_q[x]$ be an MVSP  with value set $V_F$ of size $\# V_F>2$. Then there exist a subfield $\F_{p^k}\subseteq \F_q$, a divisor $v$ of $p^k-1$, and an $\F_{p^k}$-vector space $\mathcal U\subseteq \F_q$, with $1 \in\mathcal U$  such that $V_F=a\cdot \mathcal U^v+b$ for some $a, b\in\F_q$.

\end{lemma}

 \begin{proof} 
 We use Theorem~\ref{thm:mills} and its notation. From assertion (iv), we have the polynomial $T(x^v+\gamma_0)/x^{v-1}=\sum_{i=0}^{m}w_ix^{p^{ki}}$, whose set of roots is an $\F_{p^k}$-vector space $\mathcal U_0\subseteq \overline{\F}_q$, and $V_F=\mathcal U_0^v+\gamma_0$. Now fix $u\in \mathcal U_0\backslash \{0\}$, and set $\mathcal U=u^{-1}\cdot \mathcal U_0$. This gives $1\in \mathcal U$ and $V_F=a\cdot \mathcal U^v+\gamma_0$, with $a=u^{v}\in \F_q$. It remains to prove that $\mathcal U\subseteq \F_q$.

First, note that $1\in \mathcal U$ implies $\F_{p^k}\subseteq \mathcal U$, and since $\F_{p^k}\subseteq \F_q$, we only need to prove that $\mathcal U\setminus \F_{p^k}\subseteq \F_q$. Observe that the conditions $\mathcal U^v \subseteq \F_q$ and $v\mid (p^k-1)$ give
$$y\in  \mathcal U\setminus \F_{p^k}\Longrightarrow y^v\in \F_q \Longrightarrow y^{p^k-1}\in \F_q   \Longrightarrow   y^{p^k}=c_0y$$
 for some $c_0\in \F_q\setminus \{0, 1\}$, once $y\not\in \F_{p^k}$. The condition  $1\in \mathcal U$ also gives  $y+1\in \mathcal U\setminus \F_{p^k}$, and the same argument implies $(y+1)^{p^k}=c_1(y+1)$ for some $c_1\in \F_q\setminus \{0, 1\}$. Therefore,
$$c_1(y+1)=(y+1)^{p^k}=y^{p^k}+1=c_0y+1,$$
which gives $c_1-1=(c_0-c_1)y$. As $c_1\ne 1$, we obtain $c_0\ne c_1$,  and then $y=\frac{c_1-1}{c_0-c_1}\in \F_q$.
\end{proof}

The following lemma shows that the conditions on $S$ stated in Theorem \ref{thm:main1} are also sufficient.

\begin{lemma}\label{lem:aux}
Let $\F_{p^k}$ be a subfield of $ \F_q$, $v$ a divisor of $p^k-1$, and $\mathcal U\subseteq \F_q$ an $\F_{p^k}$-vector space with $\#\mathcal U^v>2$. Then the following hold.
\begin{enumerate}[\rm(i)]
\item $F\in\mathcal P(\mathcal U, q)$ if and only if $F^v\in\mathcal P(\mathcal U^v, q)$.
\item If $S=a\cdot\mathcal U^{v}+b$, where $a, b\in \F_q$ and $a\neq 0$, then $\mathcal P(S, q) \neq \emptyset$.
\end{enumerate}
\end{lemma}

\begin{proof}
By Lemma~\ref{lem:linear},
\begin{equation}\label{polyT}
T(x)=\prod_{u\in \mathcal U}(x-u)=\sum_{i=0}^ma_ix^{p^{ki}} \in \F_q[x],
\end{equation}
where $m=\dim \mathcal U$,   $a_m=1$, and  $a_0\neq 0$. 
Hence $$T^*(x)=\prod_{w\in\mathcal U^v}(x-w)=\sum_{i=0}^ma_ix^{\frac{p^{ki}-1}{v}+1},$$
and we have the following relation:
$$T^*(F^v)=\sum_{i=0}^ma_i(F^v)^{\frac{p^{ki}-1}{v}+1}=F^{v-1}\cdot T(F)$$
for all $F \in \F_q[x]$. Applying Theorem \ref{MillsBor} to $F$ and $F^v$, we conclude the proof of assertion (i).

To prove (ii), we observe that $F\in\F_q[x]$ is an MVSP with value set $a\cdot\mathcal{U}^v+b$ if and only if $a^{-1}\cdot(F-b)$ is an MVSP with value set $\mathcal{U}^v$. Thus we may assume  that $S=\mathcal{U}^v$, and from assertion (i), it suffices to show that $\mathcal{P}(\mathcal U,q)\neq\emptyset$.

To this end, we consider the polynomial $T$  in Eq.~\eqref{polyT}. Since $T$ is a $p^k$-linearized polynomial that divides $x^q-x$, it follows from Lemma~\ref{lem:linear} that there exists another $p^k$-linearized polynomial $M\in\F_q[x]$ such that $T(M(x))=x^q-x$. We differentiate this equation to obtain $T'(M(x))\cdot M'(x)=-1$, and hence $M'(x)=c\in\F_q^*$. Consequently, we have $T(M(x))=c^{-1}(x^q-x)M'(x)$, and,  by Theorem~\ref{MillsBor}, we conclude that $M\in\F_q[x]$ is an MVSP with value set $\mathcal{U}$.

\end{proof}

\section{MVSPs whose value set is an affine subspace}

In this section, we continue working with  $p^k$-linearized polynomials. But for convenience, $p^k$ will be denoted by $q$, and our field extension will be $\F_{q^n}$.

\begin{definition}
Let $f \in \F_{q^n}[x]$ be a polynomial. We define $\m(f)$ to be the set of non-negative integers $i$ such that a monomial $ax^i$ appears in $f$. We also define $\m(f)_0$ to be the subset of $\m(f)$ consisting of all multiples of $p$  and $\m(f)_1$ to be the complement of $\m(f)_0$ in $\m(f)$.
\end{definition}

Let $A(x)=\sum_{i=0}^ka_ix^{q^{i}}$, with $a_k=1$, be a $q$-linearized polynomial  that divides $x^{q^n}-x$. In particular,  $A(x)$ is separable. We set 
\begin{equation}\label{MVSP-A}
\mathcal W(A|\F_{q^n})=\{F\in \F_{q^n}[x]\,|\, A(F(x))=-a_0(x^{q^n}-x)F'(x)\}.
\end{equation}

Note that  $a_0=A'(x)\neq0$,  and from Theorem \ref{MillsBor}, the nonconstant polynomials $F \in \mathcal W(A|\F_{q^n}) $ are MVSPs whose value set is comprised by the roots of $A$.

A first step  is the description of the set $\mathcal {W}(A|\F_{q^n})$. For this, we begin with the following result, which provides relevant information on  the monomials of $F \in \mathcal W(A|\F_{q^n})$.

\begin{lemma}\label{lem:main}
Suppose  $A(x)=\sum_{i=0}^ka_ix^{q^{i}}$ is not a binomial, and let  $F\in \mathcal W(A|\F_{q^n})$ be a nonconstant polynomial such that $F(0)=0$. Then there exist integers $t\ge 1$ and $d_1, \ldots, d_t$ such that $\m(F)_1=\{e_1, \ldots, e_t\}$, where $e_i=d_iq^k+1$. Furthermore, the following hold.

\begin{enumerate}[\rm(i)]
\item The elements of $\m(F)_0$ are all of the form $e_iq^l$ for some $i,l\ge 1$.
\item There exist   $D_1,\ldots,D_t\geq 1$  such that $$\m(A(F(x))) =\{e_1, \ldots, e_t, e_1q^{D_1}, \ldots, e_tq^{D_t}\}.$$
\end{enumerate}
\end{lemma}
\begin{proof}
From  \eqref{MVSP-A} we have 
\begin{equation}\label{eq:dif}
A(F(x))=-a_0(x^{q^n-1}-1)xF'(x).
\end{equation}
Since  $F$ is not a constant, Eq.~ \eqref{eq:dif}  yields $F'\neq 0$ and $a_0\neq0$. The former implies that  $\m(F)_1\neq \emptyset$, i.e., $\m(F)_1=\{e_1, \ldots, e_t\}$ for some $t\geq 1$.  Moreover, since $A(x)=\sum_{i=0}^ka_ix^{q^{i}}$ and $a_0\neq0$, we have   that $\m(A(F(x)))_1=\{e_1, \ldots, e_t\}$. Equation ~ \eqref{eq:dif}  clearly gives
 \begin{equation}\label{eq:set}\m(A(F(x)))=\m(-a_0(x^{q^n-1}-1)xF'(x))=\bigcup_{i=1}^t \{e_i, e_i+q^{n}-1\}.\end{equation}
 In particular, 
\begin{equation}\label{eq:set0}   \m(A(F(x)))_0=\{e_1+q^n-1, \ldots, e_t+q^n-1\}.  \end{equation}
Note that a degree comparison in \eqref{eq:dif} gives $\deg(F)\le \frac{q^n-1}{q^k-1}<q^{n-k+1}$, and then $\deg(F(x)^{q^j})<q^{n+j-(k-1)}<q^n$ for all $j=0, \ldots, k-1$. Thus   Eqs.~ \eqref{eq:dif}  and \eqref{eq:set0} entail $ \m(A(F(x)))_0\subseteq  \m (F(x)^{q^k})$, i.e., 
 for each $ e_i\in \m(F)_1 $, there exists $m_i \in \m(F)$ such that
$$e_i-1+q^n=q^k\cdot m_i.$$
In particular, $m_i\ge q^{n-k}$, and  thus  $m_i=q^{n-k}+d_i$ for some integer $d_i\geq0$.
This implies that $e_i=d_iq^k+1$ for $i=1, \ldots, t$. 

We now proceed to prove the additional assertions.
\begin{enumerate}[\rm(i)]
\item Let  $s\in \m(F)_0$. Since $s\ne e_i$ and  $s<q^n$, Eq.~\eqref{eq:set} implies $s\not\in \m(A(F(x)))$. Since $F(0)=0$, we have that $s\ne 0$, and thus the former implies that $s=q^i\cdot s_0$ for some $s_0\in \m(F)$ and some $i\ge 1$. If $s_0\in \m(F)_1$, then our assertion follows. Otherwise, we employ the same reasoning  for $s_0\in \m(F)_0$, and eventually arrive at  $s=\tilde{s}q^N$ for some  $\tilde{s}\in \m(F)_1$ and $N\ge 1$.
\item Fix $i=1,\ldots, t$, and let $h\ge 0$ be the largest integer such that $e_iq^h\in \m(F)$. In particular, the composition $A(F(x))$ produces a monomial of degree $e_iq^{h+k}$, and every other monomial produced by this composition has the form $e_jq^r$, with $j\ne i$ or $r<h+k$. In particular, monomials of degree $e_iq^{h+k}$ cannot be canceled in $A(F(x))$, and then  $e_iq^{h+k}\in \m(A(F(x))_0$. Since $\m(A(F(x))_0$ has exactly $t$ elements and this is the number of $e_i$'s, the assertion follows.
\end{enumerate}
\end{proof}

We fix the following notations that will be frequently employed.  
\begin{enumerate}
\item $A(x)=\sum_{i=0}^ka_ix^{q^{i}}$ is a monic  $q$-linearized polynomial  of degree $q^k$ that divides $x^{q^n}-x$. In particular, $A(x)$ is  separable, that is, $A'(x)=a_0\neq0$.
\item $d_A$ is the least positive integer such that $A(M(x))=ax^{q^{d_A}}+bx$ for some $a, b\in \F_{q^n}^*$ and some   $q$-linearized polynomial $M\in \F_{q^n}[x]$ (from item (iii) in Lemma~\ref{lem:linear} such integer always exists since we can use the binomial $x^{q^n}-x$).
\item $M_A(x), \alpha_A$ are unique such that $A(M_A(x))=x^{q^{d_A}}-\alpha_{A}x$.

\end{enumerate}

We obtain the following result.

\begin{proposition}\label{prop:affine}
If $A(x), d_A, M_A(x), \alpha_A$ are as above, then
\begin{equation}\label{propW}
\mathcal W(A|\F_{q^n})=\{M_A(G(x))\,|\, G\in \mathcal W(x^{q^{d_A}}-\alpha_Ax|\F_{q^n})\}.
\end{equation}
\end{proposition}
\begin{proof}
Note that the inclusion $ \supseteq $ in \eqref{propW} is straightforward. For the other inclusion, take $F\in \mathcal W(A|\F_{q^n})$ with $F$ not zero (the case $F(x)\equiv 0$ is trivial). Hence $A(F(0))=0$ and, since $A$ is $q$-linearized, we have that 
\begin{align*}
A(F(x)- F(0))&= A(F(x))- A(F(0)) \\
&=A(F(x))\\&=-a_0(x^q-x)F'\\
&=-a_0(x^q-x)(F- F(0))'.
\end{align*}
Therefore, $F(x)- F(0)\in \mathcal W(A|\F_{q^n})$. In particular, we can assume that $F(0)=0$. If $A$ is binomial, it follows by definition that $A(x)=x^{q^{d_A}}-\alpha_Ax, M_A(x)=x$ and thus the result is trivial. 

Assume now that $A$ is not a binomial.  Form assertion (i) of  Lemma~\ref{lem:main}, it follows that
\begin{equation}\label{qqcoisa}
F(x)=\sum_{i=1}^tM_i(x^{e_{i}}),
\end{equation}
where each $M_i\in \F_{q^n}[x]$ is a $q$-linearized polynomial.  From Eq.~\eqref{qqcoisa}, we have
$$A(F(x))=A\left (\sum_{i=1}^tM_i(x^{e_{i}}) \right)=\sum_{i=1}^tA\left(M_i(x^{e_{i}})\right).$$
For each $i=1, \dots, t$, set $C_i=\m(A(M_i(x^{e_i})))  \subseteq \{e_iq^s\, : s\geq 0\}$. Since $\gcd(e_i,q)=1$, it follows that the sets  $C_i$ are pairwise disjoint.
The latter combined with assertion (ii) of Lemma~\ref{lem:main} imply that 
$$A(M_i(x^{e_{i}}))=a_ix^{e_{i}}+b_ix^{e_{i}q^{D_i}},$$
for some $a_i, b_i\in \F_{q^n}^*$, and some $D_i\geq 1$.  Hence $A(M_i(\varepsilon_ix))=x^{q^{D_i}}+\delta_ix$ for $\varepsilon_i=b_i^{-q^{n-D_i}}$ and some $\delta_i\in \F_{q^n}^*$. By the minimality of $d_A$, we have that $D_i\ge d_A$. By Lemma \eqref{novo}, there exists an integer $R_i$, with  $0\le R_i<d_A$, an element $\Delta_i\in \F_{q^n}^*$,  and a $q$-linearized polynomial $T_i\in \F_{q^n}[x]$ such that 
$$x^{q^{D_i}}=T_i(x)^{q^{d_A}}-\alpha_A T_i(x)+\Delta_i x^{q^{R_i}}=A(M_A(T_i(x)))+\Delta_i x^{q^{R_i}}.$$
Hence
$$S_i(x):=A\left(M_i(\varepsilon_ix)-M_A(T_i(x))\right)=\Delta_ix^{q^{R_i}}+\delta_i x.$$
Since $R_i<d_A$ and $\delta_i, \Delta_i\ne 0$, we must have $S_i(x)=0$ by the definition of $d_A$. The latter implies that $M_i(\varepsilon_i x)=M_A(T_i(x))+c_i$ with $c_i\in \F_{q^n}$. Since such polynomials are $q$-linearized, they both vanish at $0\in \F_{q^n}$. Therefore,   $M_i(\varepsilon_ix)=M_A(T_i(x))$, and then 
$$M_i(x)=M_A(T_i(\varepsilon_i^{-1}x)) \quad \text{  for  } 1\le i\le t.$$
All in all, we have proved that $F(x) = M_A(T(x))$ for some $q$-linearized polynomial $T \in \mathbb{F}_{q^n}[x]$. Since $A(M_A(x)) = x^{q^{d_A}} - \alpha_A x =: P_{\alpha_A}(x)$ and $A'(x) = a_0$, by taking derivatives, we conclude that $a_0 \cdot M'_A(x) = -\alpha_A = P'_{\alpha_A}(x)$. In particular,
\begin{align*}
P_{\alpha_A}(T(x)) &= A(M_A(T(x)))\\ & = A(F(x)) \\
&= -a_0(x^{q^n} - x) F'(x) \\
&= -a_0(x^{q^n} - x) M_A'(x) \cdot T'(x) \\
&= \alpha_A \cdot (x^{q^n} - x) T'(x).
\end{align*}
Therefore, $T \in \mathcal{W}(x^{q^{d_A}} - \alpha_A x \mid \mathbb{F}_{q^n})$.

\end{proof}

\subsection{ Proof of Theorem~\ref{thm:main2}}
From Lemma~\ref{lem:linear}, $A(x)=\prod_{u \in \mathcal{U}}(x-u)$ is a $p^k$-linearized polynomial. Also, by construction, $A(x)$ is monic and separable, and it divides $x^q-x$. As previously, let $x^{p^{k d_A}}-\alpha_A x \in \mathbb{F}_q[x]$ be the $p^k$-linearized binomial over $\mathbb{F}_q$ of least positive degree that is divisible by $A(x)$. We must have $\alpha_A=1$, as $1 \in \mathcal{U}$ is a root of $A(x)$. Hence $A(x)$ divides $x^{p^{k d_A}}-x$, and thus  $\mathcal{U} \subseteq \mathbb{F}_{p^{kd_A}}$. By the minimality of $d$, we have that $d_A \geq d$. Conversely, it follows, by the definition of $d$, that $\mathcal{U} \subseteq \mathbb{F}_{p^{d k}}$, hence $A(x)$ divides $x^{p^{d k}}-x$. By the minimality of $d_A$, we have that $d \geq d_A$. In conclusion, we have that $d=d_A$. Recall that $k$ divides $n$, hence $n=k s$ with $s \geq 1$. Since $\mathcal{U} \subseteq \mathbb{F}_q=\mathbb{F}_{p^{k s}}$, we have that $\mathcal{U} \subseteq \mathbb{F}_{p^{k s}} \cap \mathbb{F}_{p^{k d}}=\mathbb{F}_{p^{k g}}$, where $g=\operatorname{gcd}(s, d)$. Again, by the minimality of $d$, we must have $g=d$,  and thus $d$ divides $s$. In particular,  $\mathbb{F}_{p^{k d}} \subseteq \mathbb{F}_q$.

All in all, we proved that $d=d_A$ and $x^{p^{k d_A}}-\alpha_A x=x^{p^{d k}}-x$. Since $A(x)$ divides $x^{p^{dk}}-x$, Lemma~\ref{lem:linear} entails that there exists a $p^k$-linearized polynomial $M \in \mathbb{F}_q[x]$ such that
$$
A(M(x))=M(A(x))=x^{p^{d k}}-x .
$$
From Theorem~\ref{MillsBor} and Proposition~\ref{prop:affine}, we have that
$$
\mathcal{P}(\mathcal{U}, q)=\mathcal{W}\left(A(x) \mid \mathbb{F}_q\right) \backslash \mathbb{F}_q=\left\{M(f(x)) \mid f \in \mathcal{W}\left(x^{p^{d k}}-x \mid \mathbb{F}_q\right)\right\} \backslash \mathbb{F}_q .
$$
However, $M(f(x))$ is a constant if and only if $f$ is a constant. Therefore, we obtain
\small
$$
\left\{M(f(x)) \mid f \in \mathcal{W}\left(x^{p^{d k}}-x \mid \mathbb{F}_q\right)\right\} \backslash \mathbb{F}_q=\left\{M(f(x)) \mid f \in \mathcal{W}\left(x^{p^{d k}}-x \mid \mathbb{F}_q\right) \backslash \mathbb{F}_q\right\}.
$$
From Theorem~\ref{MillsBor}, we have that
$$
\mathcal{W}\left(x^{p^{d k}}-x \mid \mathbb{F}_q\right) \backslash \mathbb{F}_q=\mathcal{P}\left(\mathbb{F}_{p^{d k}}, q\right) .
$$

\section{On the Conjecture~\ref{conj}}

In this section, we provide examples where Conjecture~\ref{conj} holds. Before doing so, we briefly discuss the conjecture itself and the fundamental idea underlying  its proof. The conjecture offers a framework for understanding MVSPs  in terms of the structure of the sets \(\mathcal{U}^v\), where \(\mathcal{U} \subseteq \mathbb{F}_q\) is a vector space. A key aspect of the conjecture is that all MVSPs are of the form \(H^s\), where \(s\) is a suitable integer and \(H\) is a simpler MVSP derived from either previous results or new ones presented in this paper. The source for \(H\) depends on whether or not \(\mathcal{U}\) is a field. Note that  this condition is consistent in the sense that if \(\mathcal{U}\) is not a field, then \(\mathcal{U}^v\) cannot be expressed as \(\mathcal{U}_1^{v_1}\) with \(\mathcal{U}_1\) being a field.
To see this, let \(\mathcal{U} \subseteq \mathbb{F}_q\) be an \(\mathbb{F}_{p^k}\)-vector space of dimension \(m\), i.e., \(\mathcal{U}\) consists of the roots of a factor of \(x^q - x\) of the form
\[
T(x) = \sum_{i=0}^m a_i x^{p^{k i}},
\]
with \(a_m = 1\) and \(a_0 \neq 0\). If \(v\) is a divisor of \(p^k - 1\), and \(\mathcal{U}^v = (\mathbb{F}_{Q})^s\) for some divisor \(s\) of \(Q - 1\), then we have the polynomial identity
\[
\sum_{i=0}^m a_i x^{\frac{p^{k i} - 1}{v} + 1} = x^{\frac{Q-1}{s} + 1} - x,
\]
which implies that \(a_0 = -1\) and \(a_i = 0\) for \(i = 1, \ldots, m-1\). Therefore,
\[
T(x) = x^{p^{m k}} - x,
\]
and thus \(\mathcal{U} = \mathbb{F}_{p^{m k}}\) is a field.

The proof of Conjecture~\ref{conj} is based on  two main steps. First, by considering an MVSP \(F\) with value set \(V_F = \mathcal{U}^v\), we derive relations between the parameters associated with the set \(\mathcal{U}^v\) and the parameters associated with the polynomial \(F\), as described in Theorem~\ref{thm:mills}. A crucial auxiliary result in this context is Lemma~\ref{lem:gamma}, presented below. Second, since Theorem~\ref{thm:main1} guarantees that \(\mathcal{U} \subseteq \mathbb{F}_q\), we conclude that \(F(y)\) is an \(s\)-th power in \(\mathbb{F}_q^*\) for a large proportion of the elements \(y \in \mathbb{F}_q\). This results in a character sum associated with \(F\) having a relatively large absolute value, and by employing Weil's bound, we conclude that \(F\) must have an extremely large degree unless it takes the form \(H^s\).

 \begin{lemma}\label{lem:gamma}
Let $\F_{p^k}$ be a subfield of $ \F_q$, $v$ a divisor of $p^k-1$, and $\mathcal U\subseteq \F_q$  an $\F_{p^k}$-vector space such that $1\in \mathcal U$ and $\#\mathcal U^v>2$. Suppose that $F\in \mathcal P(\mathcal U^v, q)$, and let $\gamma_0 \in V_F=\mathcal U^v$ be as in Theorem~\ref{thm:mills}. If $v>1$ and $\mathcal U^v$ is not a subfield of $\F_q$, then $\gamma_0=0$.
\end{lemma}

 \begin{proof}
From (i) in Lemma~\ref{lem:linear}, we know that   
$\prod_{u\in \mathcal U}(x-u)=\sum_{i=0}^{m}a_ix^{p^{ki}} \in \F_q[x],$
where  $m=\dim \mathcal U\geq 1$,  $a_m=1$, and $a_0\ne 0$. Therefore, since   $ V_F=\mathcal U^v$, we have  
$$T(x)=\prod_{a\in V_F}(x-a)=\sum_{i=0}^{m}a_ix^{\frac{p^{ki}-1}{v}+1}.$$
Using Theorem~\ref{thm:mills}, we can find a subfield $\F_{p^K}$ of $\F_q$, a divisor $w$ of $p^{K}-1$, and an integer $M\geq 1$ such that $r=\frac{p^{MK}-1}{w}=\frac{p^{mk}-1}{v}=\#V_F-1$, and
 $T(x+\gamma_0)=\sum_{j=0}^{M}b_jx^{\frac{p^{Kj}-1}{w}+1} \in \F_q[x]$, where $b_M=1, b_0\ne 0$. In particular, 
  \begin{equation}\label{eq:gamma}\sum_{i=0}^{m}a_i (x+\gamma_0)^{\frac{p^{ki}-1}{v}+1}=\sum_{j=0}^{M}b_jx^{\frac{p^{Kj}-1}{w}+1}.\end{equation} 
Since $\#V_F=r+1>2$, we have $r>1$, and then  $(M, w)\ne (1, p^K-1)$. The latter implies that $r=\frac{p^{MK}-1}{w}>\frac{p^{K(M-1)}-1}{w}+1$, and,  as a consequence, the right-hand side of Eq.~\eqref{eq:gamma} has no monomial of degree $r$.

Similarly, we have $(m, v)\ne (1, p^k-1)$, and then $\frac{p^{k(m-1)}-1}{v} +1<r$. Thus, from Eq.~\eqref{eq:gamma}, the monomial $(r+1)\gamma_0 x^r$ in the binomial expansion of $(x+\gamma_0)^{r+1}$ vanishes.

Therefore, either $\gamma_0=0$ or $r\equiv -1\pmod p$. Suppose, by way of contradiction, that $\gamma_0\ne 0$. Then $r\equiv -1\pmod p$, which implies $v, w\equiv 1\pmod p$. In particular, $\frac{p^{ki}-1}{v}+1, \frac{p^{Kj}-1}{w}+1\equiv 0\pmod p$ for every $i, j\ge1$, and $a_0=b_0$.

We write $r+1=p^t\cdot s$, where $t\ge 1$ and $\gcd(s, p)=1$. We  split the proof into two cases.

 \begin{itemize}
\item Case $m=1$. In this case, we cancel $a_0x$ and $b_0x$ in Eq.\eqref{eq:gamma} to obtain
\begin{equation}\label{eq:gamma2}(x^{p^t}+\gamma_0^{p^t})^{s}+a_0\gamma_0=\sum_{j=1}^{M}b_jx^{\frac{p^{Kj}-1}{w}+1}.\end{equation}
If $s=1$, then $r=p^t-1$ and $v=\frac{p^k-1}{p^t-1}$. Since $m=1$ and $1\in \mathcal U$, we have  $\mathcal U=\F_{p^k}$, and thus $\mathcal U^v= \F_{p^t}$ is a subfield of $\F_q$, which is a contradiction. Hence $s\geq 2$, and since $\gcd(s, p)=1$, the left-hand side of Eq.\eqref{eq:gamma2} has at least two monomials of distinct degrees. This implies that $M>1$.
 Note  that $$s\gamma_0^{p^t}x^{p^t(s-1)}=s\gamma_0^{p^t}x^{r+1-p^t} \neq 0$$ is the  monomial of second largest degree in the left-hand side of Eq.~\eqref{eq:gamma2}. This implies 
 \begin{equation}\label{eq:gamma2.1}
 r+1-p^t\le \frac{p^{K(M-1)}-1}{w}+1,
 \end{equation}
 and since $rw=p^{MK}-1$, we have that
 $p^{MK}\le p^{K(M-1)}+wp^t$, which gives 
   \begin{equation}\label{eq:gamma2.2}
 p^{t}\ge \frac{p^{K(M-1)} (p^K-1)}{w}\geq p^{K(M-1)}.
  \end{equation}
Note that  $s\geq 2$  gives  $r+1=sp^t\geq 2p^t$, and then  from  Eqs.~\eqref{eq:gamma2.1}  and  \eqref{eq:gamma2.2} we have 
 $$p^t\leq r+1-p^t\leq  \frac{p^{K(M-1)}-1}{w}+ 1\le \frac{p^{t}-1}{w}+1,$$
which implies  $w=1$, $s=2$, $t=K(M-1)$. Thus Eq. ~\eqref{eq:gamma2.2} gives $p^K=2$, i.e., $K=1$ and $p=2$, which contradicts $\gcd(s, p)=1$.

 \item Case  $m>1$. In this case, we  claim that $s\geq 2$. First note that  $vr+1=p^{mk}$  and $v \mid (p^k-1)$ give   $r\ge \frac{p^{mk}-1}{p^k-1}>p^{k(m-1)}$.
For $s=1$, we have   $r=p^t-1$  and
  $$v=\frac{p^{mk}-1}{r}=\frac{p^{mk}-1}{p^t-1}.$$
Then $v>1$ implies  that $t$ is a proper divisor of $mk$. In particular, $r=p^t-1< p^{mk/2}$.  This gives us $k(m-1) < mk/2$, which further implies $m \leq 1$, a contradiction.
Now, since $r>p^{k(m-1)}$ and $s, v\ge 2$, we have
 $$\frac{p^{k(m-1)}-1}{v}<\frac{r+1}{2}= \frac{s}{2}\cdot p^t\le (s-1)p^t=r+1-p^t.$$ 
 Therefore, $s\gamma_0^{p^t}x^{r+1-p^t}\neq 0$ is the monomial of second largest degree in the  left-hand side  of Eq.~\eqref{eq:gamma}, and this monomial is not canceled in the expansion of $\sum_{i=0}^{m}a_i(x+\gamma_0)^{\frac{p^{ki}-1}{v}+1}$. We can then proceed as  in  Eq.~\eqref{eq:gamma2.1} of item (i) to obtain a contradiction.

\end{itemize}
 \end{proof}

\begin{remark}
The conditions $v>1$ and $\mathcal{U}^v$ not being a subfield of $\F_q$ are necessary in Lemma~\ref{cor:vec}. In fact, if $v=1$, $F\in \mathcal P(\mathcal U, q)$, and $\gamma_0 \in V_F=\mathcal U^v=\mathcal U$ is as in Theorem~\ref{thm:mills}, then $F+u\in  \mathcal P(\mathcal U, q)$ for every $u\in \mathcal U$. Moreover, the corresponding $\gamma_0$ for $F+u$ equals $\gamma_0+u$. As $\#\mathcal U>2$, we can take $u\in \mathcal U$ in a way that $u+\gamma_0\ne 0$. Now for $\mathcal U^v=\F_{Q}$, a subfield of $\F_q$, we have a generic counterexample
for Lemma~\ref{cor:vec}: it is direct to verify that, for every $a\in \F_Q^*$, the polynomial $F(x)=x^{\frac{q-1}{Q-1}}+a$ is an MVSP over $\F_q$ with $V_F=\F_{Q}$ and $\gamma_0=a\ne 0$.
\end{remark}

\subsection{Proof of Theorem~\ref{pre-conj}}

\begin{proof}
Let $F\in\mathcal P(\mathcal U^v, q)$, hence $\# V_F=\# \mathcal U^v=\frac{p^{mk}-1}{v}+1$. Therefore, if $k_0, v_0,$ and  $m_0$ are the positive integers associated to $F$, as in Theorem~\ref{thm:mills}, then $v_0$ divides $p^{k_0}-1$, and $\# V_F=\frac{p^{m_0k_0}-1}{v_0}+1$. In particular, 
\begin{equation}\label{equ-r}
r=\#V_F-1=\frac{p^{mk}-1}{v}=\frac{p^{m_0k_0}-1}{v_0}.
 \end{equation}
  We consider the cases (i) and (ii) separately.
\begin{itemize}
\item Suppose 
$\mathcal U\ne \F_{p^{mk}}$. 
Since $1\in \mathcal U$, we have $\F_{p^k}\subseteq \mathcal U\ne \F_{p^{mk}}$, and then $m>1$. From Lemma~\ref{lem:linear}, $A(x)=\prod_{u\in \mathcal U}(x-u)=\sum_{i=0}^ma_ix^{p^{ki}}$, where $a_m=1$ and $a_0\ne 0$. Since $V_F=\mathcal U^v$, we obtain $T(x)=\prod_{z\in V_F}(x-z)=\sum_{i=0}^ma_ix^{\frac{p^{ki}-1}{v}+1}$. Since $\mathcal U\ne \F_{p^{mk}}$, Remark~\ref{rem} implies that $V_F=\mathcal U^v$ is not a subfield of $\F_q$. Suppose $F\not\in \{f^v\,|\, f\in \mathcal P(\mathcal U, q)\}$, hence $v\ne 1$. Therefore, if $\gamma_0$ is the element of $\F_q$ associated to $F$,  as in Theorem~\ref{thm:mills}, Lemma~\ref{lem:gamma} entails that $\gamma_0=0$. 

We claim that $m_0>1$. Suppose, by contradiction, that $m_0=1$. In this case, Theorem~\ref{thm:mills} entails that $T(x)=x^{\frac{p^{m_0k_0}-1}{v_0}+1}-\delta x$ for some $\delta\in \F_q^*$. However, since $1\in \mathcal U$, we also have $1\in V_F$. In particular, $T(1)=0$, and then $\delta=1$. In conclusion, 
$$\sum_{i=0}^ma_ix^{\frac{p^{ki}-1}{v}+1}=x^{\frac{p^{m_0k_0}-1}{v_0}+1}- x,$$
hence $a_0=-1$,  and $a_i=0$ for $i=1, \ldots,m-1$. The latter implies that $A(x)=x^{p^{mk}}-x$, and then $\mathcal U=\F_{p^{mk}}$, a contradiction. In particular, we proved that $m, m_0>1$.   We claim that $mk=m_0k_0$.  Note that Eq. \eqref{equ-r}  gives
$$(p^{mk}-1)v_0= (p^{m_0k_0}-1)v,$$
and $p^k-1$ being a multiple of  $v$ implies that $p^{mk}-1$ divides $(p^{m_0k_0}-1)(p^k-1)$. Therefore,  for $d=\gcd(mk, m_0k_0)$, we have that 
$\frac{p^{mk}-1}{p^d-1}$ divides $p^k-1$. If $d<mk$, then  $d\le \frac{mk}{2}$,  which  gives
$$\frac{p^{mk}-1}{p^d-1}>p^{mk-d}-1\ge p^{\frac{mk}{2}}-1.$$
Hence $p^k-1>p^{\frac{mk}{2}}-1$, that is,  $k>\frac{mk}{2}$, which contradicts  $m>1$. Therefore,  $d=mk$.  Since $m_0>1$,  in a similar way, we obtain $d=m_0k_0$, which yields  $mk=m_0k_0$,  as claimed. In particular, from Eq.\eqref{equ-r},   we have $v=v_0$.

 Therefore, Theorem~\ref{thm:mills} implies that, for $L=\gcd(F,x^q-x)$, there exist $A, B, N\in \F_q[x]$ such that $F=L^vN^{p^{mk}}$ and $L=A^{p^{mk}}x+B^p$. Since $F\in \mathcal P(\mathcal U^v, q)\setminus \{f^v\,|\, f\in \pp(\mathcal U, q)\}$, Lemma~\ref{lem:aux} entails that $F$ is not of the form $G^v$ with $G\in \F_q[x]$. 

Recall that $F=L^{v}N^{p^{mk}}$, $V_F=\mathcal U^v$ has at least two elements, and $\mathcal U\subseteq \F_q$. Thus  for at least one element $y\in \F_q$,  we have   $F(y)=b^v$ for some $b\in \F_{q}^*$. In particular, if $N$ is of the form $aH^v$,  with $a\in \F_q$ and $H\in \F_q[x]$, it follows that $a$ is of the form $c^v$ for some $c\in \F_q$. But this would imply  $F=L^{v}N^{p^{mk}}=G^v$ for some $G\in \F_q[x]$, a contradiction. Therefore, $N$ is a nonconstant polynomial, which  is not of the form $aH^v$ with $a\in \F_q$ and $H\in \F_q[x]$. Let $\chi_v$ be any multiplicative character of $\F_{q}$ of order $v$. Theorem~\ref{thm:Weil} entails 
\begin{equation}\label{equ-hw}
\left|\sum_{y\in \F_q}\chi_v(N(y))\right|\le (\deg(\rad(N))-1)p^{n/2}.
\end{equation}
Observe that $F(y)$ is of the form $b^v$ with $b\in \F_q^*$ for every $y\in \F_q$ such that $F(y)\ne 0$. As $F=L^{v}N^{p^{mk}}$ and $L=\gcd(F, x^q-x)$, we conclude that $\chi_v(N(y))=1$ whenever $L(y)\ne 0$. Therefore, if $Z\subseteq \F_q$ is  the set of the $\deg(L)$ distinct roots of $L$, we obtain
\begin{align*}
q - \deg(L) &= \left\lvert \sum_{y\in \F_q} \chi_v(N(y)) - \sum_{y\in Z} \chi_v(N(y)) \right\rvert \\
            &\le \left\lvert \sum_{y\in \F_q } \chi_v(N(y)) \right\rvert + \sum_{y\in Z} \lvert \chi_v(N(y)) \rvert \\
            &\le \left\lvert \sum_{y\in \F_q } \chi_v(N(y)) \right\rvert + \deg(L).
\end{align*}
Hence $ q-2\deg(L)\leq \left|    \sum_{y\in \F_q}\chi_v(N(y))\right|$, and then Eq. \eqref{equ-hw} yields 
\begin{equation}\label{rad1}
\deg(N)\geq \deg(\rad(N))\ge p^{n/2}+1-2\deg(L) p^{-n/2}.
\end{equation}
From $F=L^vN^{p^{mk}}$, it follows  that 
\begin{align*}
\deg(F) &= v \cdot \deg(L) + p^{mk} \deg(N) \\
        &\ge (p^{n/2} + 1)p^{mk} +  (v - 2p^{mk - n/2})\deg(L) .
\end{align*}
Since  $v\ge 2$,  the proof of assertion (i)  will be finished once we prove that $p^{mk-n/2}<1$, that is, $mk< \frac{n}{2}$.  Given that  $0\in V_F$ is an element with the smallest nontrivial pre-image by $F$ over $\F_q$, we obtain $\deg(L)=\#(F^{-1}(0)\cap \F_q) \le \frac{p^n}{\# V_F}=\frac{p^n}{r+1}$, and then  Eq. \eqref{rad1} gives 
\begin{equation}\label{eq:N0}
\deg(N)\ge\frac{r-1}{r+1}p^{n/2}+1.
\end{equation}
In addition, since $m>1$, $v$ divides $p^k-1$, and $r=\frac{p^{mk-1}}{v}$,  we have  
$$r\ge \frac{p^{mk}-1}{p^k-1}>p^{(m-1)k}\ge 2.$$
Hence $r\ge 3$ , $\deg(N)\ge\frac{1}{2}p^{n/2}+1$, and then  
$$\deg(F)>p^{mk}\cdot \deg(N)\ge \frac{1}{2}p^{n/2+mk}.$$
Therefore, if $mk\ge \frac{n}{2}$, then  $\deg(F)\ge \frac{q}{2}$ , which contradicts $3\le \# V_F<\frac{q}{\deg(F)}+1$.

\item Suppose  $\mathcal U=\F_{p^{mk}}$. Observe that $e$ is the smallest positive integer such that $\frac{p^{mk}-1}{v}$ divides $p^e-1$. 

From Eq. \eqref{equ-r},  we have  that 
 $\frac{p^{mk}-1}{v}$ divides $p^{m_0k_0}-1$, and then, by the minimality of $e$, it follows that $e$ divides $m_0k_0$ and  $v_0=\frac{v(p^{m_0k_0}-1)}{p^{mk}-1}$ is divisible by $t=\frac{v\left(p^e-1\right)}{p^{m k}-1}$. Moreover, the condition $\mathbb{F}_{p^e} \subseteq$ $\mathbb{F}_{p^{m k}}$ implies that $e$ divides $mk$, and  since $\F_{p^{mk}}=\mathcal U\subseteq \F_q$, we conclude that $p^e-1$ divides $q-1$.

Suppose $F\not\in\{f^t\,|\, f\in \pp(\mathcal \F_{p^e}, q)\}$. If $v=1$, then $e=mk$ and $t=1$, a contradiction with $F\in \pp(\mathcal U^v, q)=\pp(\mathcal \F_{p^e}, q)$. Hence $v>1$. 
Let $\gamma_0$ be the element of $\F_q$ associated to $F$, as in Theorem~\ref{thm:mills}. From Remark~\ref{rem}, $V_F=\mathcal U^v$ is not a subfield of $\F_q$. Hence $\gamma_0=0$ by Lemma~\ref{lem:gamma}. In particular, Theorem~\ref{thm:mills} entails that, for $L=\gcd(F,x^q-x)$, there exist $A, B, N\in \F_q[x]$ such that $F=L^{v_0}N^{p^{m_0k_0}}$ and $L=A^{p^{m_0k_0}}x+B^p$. Since $F\in \pp(\mathcal U^v, q)\setminus \{ f^{t}\,|\, f\in \pp(\F_{p^e}, q)\}$ and $v_0$ is divisible by $t$, Lemma~\ref{lem:aux} entails that $F$ is not of the form $G^t$ with $G\in \F_q[x]$. Therefore, $t>1$.

Since  $V_F=\left(\F_{p^{mk}}\right)^v$ and $p^e-1=t\cdot \frac{p^{mk}-1}{v}$ divides $q-1=p^n-1$, we have that $F(y)$ is a $t$-th power of an element in $\F_q^*$ for every $y\in \F_q$ with $F(y)\ne 0$. Following the arguments of the previous item, we conclude that $N$ satisfies Eq.~\eqref{eq:N0} and, if $r\ge 3$ or $m_0k_0\le \frac{n}{2}$, then \begin{equation}\label{eq:low}\deg(F)\ge (p^{n/2}+1)p^{m_0k_0}.\end{equation} 
However, if $r=2$,  then $2=r=\frac{p^{m_0k_0}-1}{v_0}$. As $v_0$ divides $p^{k_0}-1$, we conclude that $m_0=1$. Recall that $k_0$ divides $n$, hence $k_0=n$ or $k_0\le \frac{n}{2}$. The case $k_0=n$ cannot occur since, from hypothesis, $N$ is nonconstant and $F=L^vN^{p^{m_0k_0}}$ has degree smaller than $p^n$. Hence $m_0k_0=k_0\le \frac{n}{2}$. 
In particular, we conclude that $\deg(F)\ge (p^{n/2}+1)p^{m_0k_0}$. Since $e$ divides $m_0k_0$, we obtain $e\le m_0k_0$, and thus $\deg(F)\ge (p^{n/2}+1)p^{e}$. 
\end{itemize}
The statements regarding instances where Conjecture~\ref{conj} holds follow directly from the fact that if $F$ is an MVSP with $\# V_F=r+1$, then $\deg(F)<\frac{q}{r}$ and, if $m>1$, then $r=\frac{p^{mk}-1}{v}\ge \frac{p^{mk}-1}{p^k-1}>p^{(m-1)k}$.
\end{proof}

\subsection{MVSPs over fields of size $p^4$}
We have the following result.

\begin{theorem}\label{thm:p4-a}
 Conjecture~\ref{conj} holds for $q=p^4$. 
\end{theorem}

\begin{proof}
Let $k$ be a divisor of $4$, $v>1$  a divisor of $p^k-1$, and  $\mathcal U\subseteq \F_{p^4}$  an $\F_{p^k}$-vector space of dimension $m$ such that $(m, v)\ne (1, p^k-1)$ and $1\in\mathcal U$.

If $\mathcal U\neq \F_{p^{mk}}$, then from  (i) and (a) of Theorem~\ref{pre-conj},  Conjecture~\ref{conj} holds for $q=p^4$.  Let us assume  $\mathcal U=\F_{p^{mk}}$, and let $e, t$ be as in assertion (ii) of Theorem~\ref{pre-conj}. Suppose 
\begin{equation}\label{eq-prov}
F\in \mathcal P(\mathcal U^v, p^4)\setminus \{ f^{t}\,|\, f\in \pp(\F_{p^e}, q)\},
\end{equation}
 and let $k_0, v_0,$ and  $m_0$ be the positive integers associated to $F$, as in Theorem~\ref{thm:mills}. Moreover, set $r=\# V_F-1\ge 2$, that is, $r=\frac{p^{mk}-1}{v}$. Observe that any MVSP $F\in \F_{p^4}[x]$ with $\#V_F>1$ satisfies $\deg(F)\le p^4$, and 
recall that $\deg(F)\ge (p^{n/2}+1)p^{m_0k_0}$ was proved in Theorem~\ref{pre-conj} (see Eq. \eqref{eq:low}). Hence $m_0=k_0=1$. We have that $\frac{p^{mk}-1}{v}=\frac{p^{m_0k_0}-1}{v_0}=\frac{p-1}{v_0}$,  and thus $e$ is the least positive integer such that $\frac{p-1}{v_0} $ divides $p^e-1$, that is, $e=1$. In particular, $t=\frac{(p^e-1)v}{p^{mk}-1}=\frac{(p^e-1)v_0}{p^{m_0k_0}-1}=v_0$,  hence $rt+1=p$. Considering \eqref{eq-prov}, assertion (ii) of  Theorem \ref{pre-conj} gives $t\ge 2$, and then  $ r\ge 2$  implies $p\ge 5$. 

From Theorem~\ref{thm:mills}, there exist polynomials $A, B, L, N\in \F_{p^4}[x]$ such that the following hold:
\begin{itemize}
\item $L=\gcd(F,x^q-x)=A^px+B^p$;
\item $F=L^tN^p$;
\item  $rt+1=p$.
\end{itemize}
Note that  we are using that $\gamma_0=0$, which  follows from Lemma~\ref{lem:gamma}, as  (ii) of Theorem~\ref{pre-conj} gives that $\mathcal{U}^v$  is not a subfield of $\mathbb{F}_q$ and  $v_0=t>1$.

Eq. \eqref{eq-prov} and  Lemma~\ref{lem:aux} imply  that $F$, hence $N$, is not of the form $H^t$ with $H\in \F_{p^4}[x]$. 
Observe that  if  every root of $N$ is of multiplicity $\equiv 0, t\pmod p$, then $N=N_1^tN_2^p$, where $N_2$ is not of the form $H^t$. Following the proof of  Eqs. \eqref{rad1} and 
\eqref{eq:N0}   in the proof of Theorem~\ref{pre-conj}, we have that  $N_2$ is not of the form $aH^t$ with $a\in \F_q, H\in \F_q[x]$ and $$\deg(N_2)\ge \max\left\{p^2+1-2\deg(L)p^{-2}, \frac{r-1}{r+1}p^2\right\}.$$
The latter implies 
 $$\deg(F)>p\cdot \deg(N)\ge p^2\deg(N_2)\ge \frac{r-1}{r+1}p^4.$$ 
  Since $F$ is an MVSP with a value set of size $r+1\ge 3$, we have $\deg(F)<\frac{p^4}{2}$, which  contradicts the inequality  $\deg(F)>\frac{r-1}{r+1}p^4$ if $r>2$. If $r=2$, then $t=\frac{p-1}{2}$, since $rt+1=p$. As $p^2\deg(N_2)<\deg(F)<\frac{p^4}{2}$, we conclude that $\deg(N_2)<\frac{p^2}{2}$. Since  $\deg(N_2)\ge p^2-2\deg(L)p^{-2}+1$, we obtain $\deg(L)>\frac{p^4}{4}$. Hence $\deg(F)\ge t\cdot \deg(L)>\frac{p^4(p-1)}{8}\ge \frac{p^4}{2}$, since $p\ge 5$, a contradiction with  $\deg(F)<\frac{p^4}{2}$.  
  
  Recall that $p=rt+1$ with $r, t\ge 2$.  Hence,  once we prove that any root of $N$ has multiplicity $\equiv 0, t\pmod p$, we conclude that Conjecture~\ref{conj} holds for $q=p^4$.  If the derivative $N'$ is vanishing, this is trivially true. Assume that $N'$ is not vanishing.

Under our hypothesis, $V_F=(\mathbb{F}_p)^t=\mathcal C_r\cup\{0\}$, where $\mathcal C_r\subseteq \F_{p}$ is the cyclic group of order $r$. Hence  $T(x)=\prod_{s\in \mathcal V_F}(x-s)=x^{r+1}-x$, and thus $T'(x)=(r+1)x^r-1$. In particular, $T'(a)=r$ for every $a\in V_F\setminus \{0\}$. Since $tr=-1\in \F_p$ and $\# V_F>2$, Theorem~\ref{thm:mills} implies that $$F^{r+1}-F=T(F)=-r(x^{p^4}-x)F'=(x^{p^4}-x)L^{t-1}N^p,$$ and thus replacing $F$ with  $L^tN^p$ yields
\begin{equation}\label{eq:main}\underbrace{(A^{p}\cdot x+B^{p})N^r}_{P_1}=\underbrace{Ax^{p^{3}}+B}_{P_2}.\end{equation}
Observe that for every $b\in \F_{p^4}$, the degree and the value set of $F(x+b)$ coincide with those  of $F$. Moreover, the roots of $N$ have multiplicity $\equiv 0, t\pmod p$ if and only if the same holds for $N(x+b)$. Thus we may assume that $F(0)\ne 0$. Since $L=\gcd(F, x^q-x)=A^px+B^p$, we have $\gcd(A, B)=\gcd(B, x)=1$. In particular, Eq. \eqref{eq:main} implies  \begin{equation}\label{eq:gcd}\gcd(A, N)=\gcd(A, B)=\gcd(B, N)=\gcd(N, x)=\gcd(B, x)=1.\end{equation} If $W(P, Q)=P'Q-PQ'$ denotes the Wronskian map, Eq.\eqref{eq:main} entails that $W(P_1, N^r)=W(P_2, N^r)$, and then
\begin{equation}\label{eq:main2} A^{p}N^{r+1}=S\cdot x^{p^{3}}+T,\end{equation}
where $S=A'N-rAN'$ and $T=B'N-rBN'$. 
Since $N'\neq 0$ and $\gcd(N,A)=\gcd(N,B)=1$, we have that neither $S$ nor $T$ is vanishing. Moreover, since  $\gcd(N,xA)=1$, we have  $$G=\gcd(S x^{p^{3}}, T)=g_1g_2,$$ where $g_1=\gcd(N^{r+1}, T, S)=\gcd(N, N')$ and $g_2=\gcd\left(\frac{Sx^{p^3}}{g_1}, \frac{T}{g_1}, A^p\right)=\gcd\left(\frac{Sx^{p^3}}{g_1}, A^p\right)$, in view of Eq. \eqref{eq:gcd}. Since $g_2$ divides $A^p$, Eq.~\eqref{eq:main2} yields
\begin{equation}\label{eq:main3}\underbrace{\frac{A^{p}}{g_2}\cdot\frac{N^{r+1}}{g_1} }_{R}=\underbrace{\frac{Sx^{p^{3}}}{g_1g_2}}_{S_1}+\underbrace{\frac{T}{g_1g_2}}_{T_1}.\end{equation} 
By construction, the polynomials $R, S_1$, and $T_1$ are pairwise relatively prime. We claim  that $R$ is a polynomial of vanishing derivative.
Our proof relies on Theorem~\ref{abc}.  Therefore, if we set  $$\Delta=\max\{\deg R, \deg S_1, \deg T_1 \}\quad\text{and}\quad \varepsilon=\deg \rad(RS_1T_1),$$
it suffices to prove that $\Delta>\varepsilon-1$. Set $n_1=\deg(N)>0$ and $d=\max\{\deg(A), \deg(B)\}$.
Write $A(x)=x^{\ell}\cdot A_0(x)$, where $0\le \ell\le d$, and $A_0\in \F_{p^4}[x]$ is not divisible by $x$. 
Note that if $\ell \geq p^2$, then Eq.~\eqref{eq:main2} implies $T=B'N-rBN'\neq 0$ is divisible by $x^{p^3}$, and then  $\deg B+ \deg N >p^3$, which gives
$$\deg F\geq p(\deg B+\deg N)>p^4,$$
a contradiction. Therefore, $\ell < p^2$. From Eq.~\eqref{eq:gcd}, $N(0)\ne 0$, hence $g_1$ is not divisible by $x$, and thus $\deg(\rad(R))\le d+n_1-\ell+1$. Since  $\ell p<p^3$, it follows that   $g_2(x)$ is of the form $x^{p\ell}\cdot G_2(x)$ with $G_2\in \F_{p^4}[x]$ not divisible by $x$. Therefore, $\deg(\rad(S_1))\le  \deg(S_1)-(p^3-p\ell)+1$ and 
$\deg(T_1)\le \deg(T)-\deg(g_2) \leq d+n_1-1-p\ell$. We conclude that
$$\varepsilon-1\leq  d+n_1-\ell+\deg(S_1)-(p^3-p\ell)+d+n_1-p\ell.$$
Since $\Delta\ge \deg(S_1)$, it suffices to prove that $2d+2n_1< p^3$, which follows from $p\geq 2$ and $p(d+n_1)\leq \deg F<p^4$. Therefore, $\Delta>\varepsilon-1$.
From Theorem~\ref{abc}, $R$ is a $p$-th power. Since $\gcd(A, N)=1$, the same holds for $N_0=\frac{N^{r+1}}{\gcd(N, N')}$, and thus every root of $N_0$ has multiplicity divisible by $p$. Let $\alpha$ be a root of $N$ with multiplicity $j\not\equiv 0\pmod p$. Hence its multiplicity in $N_0$ equals $j(r+1)-(j-1)=jr+1$, forcing that $jr\equiv -1\pmod p$. Since $tr+1=p$, we obtain  $j\equiv t\pmod p$.
\end{proof}

\subsection{Proof of Theorem~\ref{thm:p4}} Using Theorem~\ref{MillsBor} and Lemma~\ref{lem:aux}, one can easily check that all polynomials  $F\in  \F_{p^4}[x]$  listed  in  Theorem~\ref{thm:p4} are MVSPs  with $\# V_F\geq 3$.   Let $F \in \F_{p^4}[x]$ be an MVSP with $\#V_F\geq 3$.
From Theorem~\ref{thm:main1}, we can assume that $V_F=\mathcal U^v$, where $\mathcal U\subseteq \F_{p^4}$ is an $\F_{p^k}$-vector space of dimension $m\ge1$ such that $1\in\mathcal U$, $k|4$, $v|p^k-1$, and $(m, v)\ne (1, p^k-1)$. In particular, $mk\le 4$. We  divide the proof into two cases.
 
\begin{enumerate}[1.]
\item $v=1$. We have the following subcases:

\begin{enumerate}
\item $\mathcal U=\F_{p^e}$ with $e\in \mathcal \{1, 2, 4\}$. The cases $e=1$ and $e=2$ are covered by items (ii) and (iii) listed in  Theorem~\ref{thm:p4}, respectively. If $e=4$, then $\deg(F)=1$. In this case, up to post-compositions, we have $F(x)=x$, which is covered by item (i).

\item $m=2$, $k=1$, and $\mathcal U\ne \F_{p^2}$. In this case, $\F_{p^4}$ is the smallest subfield containing $\mathcal U$. Let $\{1, \beta\}$ be an $\F_p$-basis of $\mathcal U$, hence $\beta\in \F_{p^4}\setminus \F_p$ and 
\begin{align*}
A(x) &= \prod_{u\in \mathcal U}(x-u) = \prod_{i, j\in \mathbb{F}_p}(x-i-j\beta) \\
     &= x^{p^2}-(1+(\beta^p-\beta)^{p-1})x^p+(\beta^p-\beta)^{p-1}x.
\end{align*}
In this case, it is direct to verify that $M(A(x))=x^{p^4}-x$, where $M(x)=x^{p^2}+(1+(\beta^{p^3}-\beta^{p^2})^{p-1})x^p-(\beta^p-\beta)^{1-p}x$. According to Theorem~\ref{thm:main2}, we have $F(x)=M(h(x))$, where $h\in \F_{p^4}[x]$ is an MVSP with value set $\F_{p^4}$. In particular, $h(x)$ has degree one. Therefore, up to pre-compositions with affine maps, we have that  $h(x)=x$, and thus $F(x)=M(x)$, which is covered by item (iv).

\item $m=3$ and $k=1$.  In this case, $\F_{p^4}$ is the smallest subfield containing $\mathcal U$. If $A(x)=\prod_{u\in \mathcal U}(x-u)$, Theorem~\ref{thm:main2} entails that there exists a $p$-linearized polynomial $M\in \F_{p^4}[x]$ such that $A(M(x))=x^{p^4}-x$. The polynomial $M$ is clearly monic, and since $m=3$, we have that $M(x)$ has degree $p$. Hence $M(x)=x^p-\alpha x$ for some $\alpha\in \F_{p^4}^*$. Moreover, since $A(M(x))=x^{p^4}-x$, Lemma~\ref{lem:linear} entails that $x^p-\alpha x$ divides $x^{p^4}-x$, and thus $\alpha=\beta^{p-1}$ for some $\beta\in \F_{p^4}^*$. According to Theorem~\ref{thm:main2}, we have that $F(x)=M(h(x))$, where $h\in \F_{p^4}[x]$ is an MVSP with value set $\F_{p^4}$. In particular, $h(x)=ax+b$ with $a, b\in \F_{p^4}$ and $a\ne 0$. Hence $F(x)=(ax+b)^p-\beta^{p-1}(ax+b)$. Therefore, up to pre- and post-compositions with affine maps, we have  $F(x)=x^p-x$. The latter is covered by item (v).

\end{enumerate}

\item $v>1$. We have the following subcases:

\begin{enumerate}
\item $\mathcal U$ is a subfield of $\F_{p^4}$. Let $e\in \{1, 2, 4\}$ be the smallest positive integer such that $\mathcal U^v\subseteq \F_{p^e}$, that is, $\frac{p^{mk}-1}{v}$ divides $p^e-1$. We also set $t=\frac{v(p^e-1)}{p^{mk}-1}$. Observe that $t$ is a divisor of $p^e-1$. From Theorem~\ref{thm:p4-a}, it follows that $F(x)=f(x)^t$, where $f\in \F_{p^4}[x]$ is an MVSP with value set $V_f=\F_{p^e}$. The cases $e=1, 2$ are covered by items (ii) and (iii) listed in  Theorem~\ref{thm:p4}, respectively. If $e=4$, we have that $f$ is an MVSP with value set $V_f=\F_{p^4}$, and thus $\deg(f)=1$. In particular, up to pre- and post-compositions with affine maps, we have $f(x)=x$. Hence $F(x)=x^v$, which is covered by item (i). 

\item $\mathcal U$ is not a subfield of $\F_{p^4}$. If $k=2$, then the conditions $\F_{p^2}\subseteq \mathcal U \subseteq \F_{p^4}$  and $\#\mathcal U=p^{2m}$ imply that  either $\mathcal U=\F_{p^2}$ or  $\mathcal U=\F_{p^4}$, a contradiction.  Therefore, $k=1$ and $m=2, 3$. As in  the previous item, we use Theorem~\ref{thm:p4-a} and conclude that $F$ is contained in the family of item (iv), for $m=2$, and in the family of item (v), for $m=3$. 

\end{enumerate}

\end{enumerate}

\subsection{Proof of Theorem~\ref{FNC}}

\begin{proof}

Let \( d = \frac{q - 1}{p^e - 1} \), and suppose \( f(x) \in \mathbb{F}_{q}[x] \) is an MVSP whose value set is \( \mathbb{F}_{p^e}     \subseteq  \mathbb{F}_{q} \). If the polynomial \( y^d - f(x) \) is irreducible, then the \(\mathbb{F}_{q}\)-Frobenius nonclassicality of the curve \( \mathcal{F}: y^d = f(x) \) follows from \cite[Corollary 3.5]{Bor0}. This is because both \( f(x) \) and \( y^d \) are MVSPs over \( \mathbb{F}_{q} \) with the same value set \( S = \mathbb{F}_{p^e} \).

Conversely, assume that the curve \( y^d = f(x) \) is \(\mathbb{F}_{q}\)-Frobenius nonclassical. By \cite[Corollary 3.5]{Bor0}, it follows that both \( y^d \) and \( f(x) \) are MVSPs over \( \mathbb{F}_{q} \) with the same value set \( S \). In particular,  \( d \) divides \( q - 1 \), and \( S = \left( \mathbb{F}_{q} \right)^d \). Note that the condition $d<q-1$ gives $\#S>2$. Assuming that Conjecture~\ref{conj} is true,  since \( \mathbb{F}_{q} \) is a field, we have that
\[
f(x) \in \mathcal{P}\left( \left( \mathbb{F}_{q} \right)^d, q \right) = \left\{ g^{\frac{\left( p^e - 1 \right) d}{q - 1}} \mid g \in \mathcal{P}\left( \mathbb{F}_{p^e}, q \right) \right\},
\]
where \( \mathbb{F}_{p^e} \subseteq \mathbb{F}_{q} \) is the smallest field containing \( \left( \mathbb{F}_{q} \right)^d \). Hence, the curve \( \mathcal{F} \) is given by
\[
y^d = g(x)^{\frac{\left( p^e - 1 \right) d}{q - 1}}  \quad \text{for some } g \in \mathcal{P}\left( \mathbb{F}_{p^e}, q \right).
\]
Since \( d = \frac{q - 1}{p^e - 1} \cdot \frac{\left( p^e - 1 \right) d}{q - 1} \) and \( \mathcal{F} \) is irreducible, we conclude that \( \frac{\left( p^e - 1 \right) d}{q- 1} = 1 \). Thus  \( \mathcal{F} \) has the equation
\[
y^{\frac{q - 1}{p^e - 1}} = f(x),
\]
where \( f(x) \in \mathbb{F}_{q}[x] \) is an MVSP with value set \(  \mathbb{F}_{p^e} \subseteq \mathbb{F}_{q} \).

\end{proof}

\section{A Final Remark}

A natural extension of the MVSP problem is the characterization of rational functions $h=f / g \in \mathbb{F}_q(t)$ whose value set attains the minimal possible size, namely $\left\lceil\frac{q+1}{\operatorname{deg} h}\right\rceil$. This fundamental question is closely connected to whether the field extension $\mathbb{F}_q(x) / \mathbb{F}_q(h(x))$ is Galois and to the structure of its Galois group. It is also related to the geometry and arithmetic of the curves defined by
$$
f_1(x) g_2(y)-g_1(x) f_2(y)=0,
$$
where $f_i / g_i \in \mathbb{F}_q(t)(i=1,2)$ are rational functions with the same minimal value set \cite{BBQ21}.

Despite the clear relationship between the polynomial and rational-function cases, the latter introduces additional subtleties, particularly in the treatment of poles. Nevertheless, a complete understanding of MVSPs is a necessary step toward resolving the rational-function problem. We believe that the framework and techniques developed here provide a solid foundation for extending the theory to rational functions, thereby establishing concrete links to broader questions in arithmetic and geometry.



\section*{Acknowledgements}
We thank the anonymous reviewers for their thoughtful feedback and valuable suggestions, which have significantly improved this work.

\end{document}